\documentclass[12pt]{amsart}
\usepackage[margin=1in]{geometry}                % See geometry.pdf to learn the layout options. There are lots.
\usepackage{graphicx}
\usepackage{amssymb}
\usepackage{amsthm}
\usepackage{mathrsfs}
\usepackage{caption}

\newtheorem{theorem}{Theorem}
\newtheorem{definition}{Definition}

\newtheorem{lemma}{Lemma}

\newcommand{\ZZ}{\mathbb{Z}}

\newcommand{\NN}{\mathbb{N}}

\newcommand{\calp}{\mathcal{P}}

\newcommand{\lcm}{{\mathrm{lcm}}}
\newcommand{\diam}{{\mathrm{diam}}}

\begin{document}

\title{On the minimal period of integer tilings} 
\author{Izabella {\L}aba and Dmitrii Zakharov}

\date{\today}

\begin{abstract}
If a finite set $A$ tiles the integers by translations, it also admits a tiling whose period $M$ has the same prime factors as $|A|$. We prove that the minimal period of such a tiling is bounded by $\exp(c(\log D)^2/\log\log D)$, where $D$ is the diameter of $A$. In the converse direction, given $\epsilon>0$, we construct tilings whose minimal period has the same prime factors as $|A|$ and is bounded from below by $D^{3/2-\epsilon}$. We also discuss the relationship between minimal tiling period estimates and the Coven-Meyerowitz conjecture.

MSC subject codes: 05B45, 52C22.
\end{abstract}

\maketitle

\section{Introduction}

%%%%%%%%%%%%%%%%%%%%%%%%%%%%%%%%%%%%

A finite set $A\subset\ZZ$ tiles the integers by translations if there exists a covering of $\ZZ$ by pairwise disjoint translates of $A$. More formally, there exists a translation set $T\subset\ZZ$ such that every integer $n\in\ZZ$ has a unique representation $n=a+t$ with $a\in A$ and $t\in T$. 

Newman \cite{New} proved that any tiling of $\ZZ$ by a finite set $A$ must be periodic, so that $T=B\oplus M\ZZ$ for some $M\in\NN$ and $B\subset \{0,1,\dots,M-1\}$. His pigeonholing argument shows that the least period of the tiling must satisfy $M\leq 2^D$, where 
$$
D:=\diam(A)=\max(A)-\min(A)
$$
is the diameter of $A$. This bound was subsequently improved by Kolountzakis \cite{Kol}, Ruzsa \cite[Appendix]{Tij2}, and Bir\'o \cite{Biro}. The best estimate to date is due to Bir\'o, who proved the following: for any $\epsilon>0$ there exists a number $D(\epsilon)$ such that if $A\oplus T = \ZZ$ is a tiling and $D=\diam(A)\geq D(\epsilon)$, then the least period of $T$ is at most $\exp(D^{1/3+\epsilon})$.

In the other direction, Kolountzakis \cite{Kol} constructed tilings with the least period $M$ bounded from below by $cD^2$ for some absolute constant $c$. Steinberger \cite{Steinberger2} gave an improved construction with 
$$M\geq \exp((\log D)^2/4\log\log D);$$
 in particular, there exist tilings whose least period is superpolynomial in $D$.
 
 In the sequel, we will always assume that $A\subset\ZZ$ is finite and nonempty, and we will always use $D$ to denote the diameter of $A$.

\begin{definition}
Assume that $A$ tiles $\ZZ$ by translation. We define the {\em minimal tiling period of $A$} to be the smallest number $\calp(A)$ such that there exists a $\calp(A)$-periodic set $T\subset\ZZ$ satisfying $A\oplus T=\ZZ$.
\end{definition}

If $A\oplus T$ is a tiling and $M$ is the least period of $T$, we clearly have $\calp(A)\leq M$. The inequality may be strict, since there may exist a different tiling of $\ZZ$ whose least period is smaller. For example, we have the following reduction, due to Coven and Meyerowitz \cite{CM} and based on the dilation theorem of Tijdeman \cite{Tij}.

\begin{lemma}\label{CM-reduction}\cite[Lemma 2.3]{CM} 
Assume that $A$ tiles $\ZZ$ with period $M$, and suppose that $M=mM'$, where $M'$ has the same prime factors as $|A|$ and $m$ is relatively prime to $|A|$. Then $A$ also tiles $\ZZ$ with period $M'$.
\end{lemma}

The tilings constructed in \cite{Kol} and \cite{Steinberger2} have long periods $M$, but in both cases we have $M=m|A|$ with $m$ relatively prime to $A$. It follows that 
$A$ also admits the tiling $A\oplus |A|\ZZ=\ZZ$, so that the minimal tiling period for $A$ is $\calp(A)=|A|\leq D+1$.

Our first result is an upper bound on $\calp(A)$ in terms of $D$.

\begin{theorem}\label{upper-bound}
Assume that $A\oplus T=\ZZ$ is a tiling, and that the least period $M$ of $T$ has the same prime factors as $|A|$. (By Lemma \ref{CM-reduction}, if $A$ tiles $\ZZ$, then such a tiling always exists.) Then
$$
\calp(A)\leq M\leq \exp(c(\log D)^2/\log\log D)$$
for some absolute constant $c>0$. 
\end{theorem}

One might ask whether the condition on the prime factors of $M$ prevents the tiling from having long periods altogether. We provide a partial answer in Theorem \ref{long-periods}. Our lower bound on the tiling period is worse than that in \cite{Kol} and especially in \cite{Steinberger2}, but, unlike in \cite{Kol} and \cite{Steinberger2}, $M$ does not have any new prime factors that are not already present in $|A|$. For a discussion of the formal similarity between our upper bound in Theorem \ref{upper-bound} and Steinberger's lower bound in \cite{Steinberger2}, see the remark at the end of Section 2.

\begin{theorem}\label{long-periods}
For any $0<\beta<3/2$, there exists a tiling $A\oplus T=\ZZ$ with least period $M$ such that $M\geq D^\beta$ and 
$M$ has the same prime factors as $|A|$.

\end{theorem}

We do not have examples of sets that tile {\it only} with long periods. The set $A$ constructed in the proof of Theorem \ref{long-periods} also tiles with period $|A|$, as do those in \cite{Kol} and \cite{Steinberger2}. Intuitively, a ``simple" tile $A$ allows significant freedom in the choice of the tiling complement, ranging from simple to more complicated. (For an analogue of this in metric geometry, one may think of cube tilings in high-dimensional spaces.)

We further note that if a tile $A$ satisfies the Coven-Meyerowitz conditions (T1) and (T2) \cite{CM}, it admits a tiling with period at most $2D$. Thus any tile with $\calp(A)>2D$ would also have to provide a counterexample to the Coven-Meyerowitz conjecture. We discuss this connection in Section \ref{CM-sec}.

\bigskip

{\em Acknowledgements.} The authors are grateful to Rachel Greenfeld for stimulating discussions. The first author was supported by NSERC Discovery Grant 22R80520.
The second author 
was supported by the Jane Street Graduate Fellowship.

\section{Proof of Theorem \ref{upper-bound}}\label{sec-thm1}
\label{upper-bound-sec}

We will need the polynomial formulation of integer tilings. Assume that $A\oplus T=\ZZ$ and that $T$ is $M$-periodic, so that $T=B\oplus M\ZZ$ for some finite set $B\subset\ZZ$. Then $A\oplus B$ mod $M$ is a factorization of the cyclic group $\ZZ_M$. We write this as $A\oplus B=\ZZ_M$. 

By translational invariance, we may assume that
$A,B\subset\{0,1,\dots\}$ and that $0\in A\cap B$. The {\it mask polynomials} of $A$ and $B$ are
$$
A(X)=\sum_{a\in A}X^a,\ B(X)=\sum_{b\in B}X^b .
$$
Then $A\oplus B=\ZZ_M$ means that
\begin{equation}\label{poly-e1}
A(X)B(X)=1+X+\dots+X^{M-1}\ \mod (X^M-1).
\end{equation}
 We may rephrase this in terms of cyclotomic polynomials as follows. Recall that the $s$-th cyclotomic polynomial $\Phi_s$ for $s\in\NN$ is the minimal polynomial of $e^{2\pi i/s}$. Alternatively, $\Phi_s$ may be defined inductively via the identity 
$$
X^N-1=\prod_{s|N} \Phi_s(X).
$$
Then (\ref{poly-e1}) is equivalent to 
\begin{equation}\label{poly-e2}
  |A||B|=M\hbox{ and }\Phi_s(X)\, |\, A(X)B(X)\hbox{ for all }s|M,\ s\neq 1.
\end{equation}
Since $\Phi_s$ are irreducible, each $\Phi_s(X)$ with $s|M$ must divide at least one of $A(X)$ and $B(X)$.

\begin{lemma}\label{top-powers}
Assume that $A\oplus T=\ZZ$ is a tiling with the least period $M$, and let $M=p_1^{n_1}\dots p_d^{n_d}$ be the prime factorization of $M$. Then for every $j\in\{1,\dots,d\}$ there exists $s|M$ such that $\Phi_s|A$ and $p_j^{n_j}|s$.
 \end{lemma}
 
\begin{proof}
We argue by contradiction. Write $T=B\oplus M\ZZ$ for some $B\subset\{0,1,\dots,M-1\}$, so that $A\oplus B=\ZZ_M$. 
Suppose that there is a $j\in\{1,\dots,d\}$ such that 
$$
\hbox{ if }s|M\hbox{ and }p_j^{n_j}|s, \hbox{ then }\Phi_s\nmid A.
$$
It follows that
$$
P(X):= \frac{X^M-1}{X^{M/p_j}-1} = \prod_{s|M,\, s\nmid\frac{M}{p_j}} \Phi_s(X)
$$
divides $B(X)$. Hence
$$
B(X)=P(X)B_0(X)=\left(1+X^{M/p_j}+X^{2M/p_j}+\dots+X^{(p_j-1)M/p_j}\right)B_0(X)
$$
for some polynomial $B_0\in \ZZ[X]$.  If $B_0(X)$ had degree $M/p_j$ or higher, then $B(X)$ would have degree $M$ or higher, contradicting the assumption that $B\subset\{0,1,\dots,M-1\}$. It follows that $B_0(X)$ is the mask polynomial of a set $B_0\subset\{0,1,\dots,(M/p_j)-1\}$, and that $B$ is $(M/p_j)$-periodic, contradicting the minimality of $M$.
\end{proof}

Recall that the degree of the cyclotomic polynomial $\Phi_s$ is equal to $\varphi(s)$, the Euler totient function. 
By Lemma \ref{top-powers}, for each $j\in\{1,\dots,d\}$ there is an $s_j$ such that $p_j^{n_j}\mid s_j$ and $\Phi_{s_j}|A$, so that
\[
D=\deg A(X) \geq  \deg \Phi_{s_j} = \varphi(s_j) \ge \varphi(p_j^{n_j}) =  (1-p_j^{-1}) p_j^{n_j} 
 \ge  p_j^{n_j}/2.
\]
Taking the product over $j\in \{1, \ldots, d\}$ leads to an upper bound $M \le (2D)^d$. 

On the other hand, since $|A|$ and $M$ have the same prime factors, we have
\[
D \ge |A|-1 \ge p_1\ldots p_d-1 \ge d! - 1,
\]
where we crudely lower-bounded the size of the $j$-th prime number by $j$. This readily implies that 
\begin{equation}\label{d-bound}
d \le c \frac{\log D}{\log \log D}
\end{equation}
for some absolute constant $c>0$. Plugging this into the upper bound on $M$ finishes the proof.

\bigskip
\noindent
{\bf Remark:} Our upper bound in Theorem \ref{upper-bound}, and Steinberger's lower bound in \cite{Steinberger2}, have the same form because they are both based on the same estimate for prime numbers. Specifically, our proof above uses that if $p_1,\dots,p_d$ are primes such that $p_1\dots p_d\leq D$, then $d$ obeys the bound (\ref{d-bound}). We then combine this with $M\geq (2D)^d$ to get our conclusion.
Steinberger optimizes the estimate (\ref{d-bound}) and constructs the set $A$ so that $|A|=p_1\dots p_d$ and $|A|\leq D$, but then the period of the tiling he constructs is divisible by additional $d$ primes $q_1,\dots,q_d$ of size about $D$; thus $M\geq D^d$ in his construction. This leads to a lower bound of the same general form as our upper bound in Theorem \ref{upper-bound}, but we were not able to replicate his construction without introducing the additional large primes not appearing in $|A|$.

%%%%%%%%%%%%%%%%%%%%%%%%%%%%

\section{Proof of Theorem \ref{long-periods}}

Let $p_1,p_2,p_3$ be large primes such that $p_1<p_2<p_3<2p_1$. Let $M=(p_1p_2p_3)^n$ for some large $n$ to be determined later. Define
$$
A(X)=\prod_{i=1}^3 \left( 1+ X^{M/p_i^n}+X^{2M/p_i^n}+\dots +X^{(p_i-1)M/p_i^n}\right),
$$
and $B_0(X)=\prod_{i=1}^3 B_i(X)$, where
$$
B_i(X)= 1+ X^{M/p_i^{n-1}}+X^{2M/p_i^{n-1}}+\dots +X^{(p_i^{n-1}-1)M/p_i^{n-1}}.
$$
Then $A\oplus B_0=\ZZ_M$ and $M$ has the same prime factors as $|A|$, but this tiling does not meet the conditions of Theorem \ref{long-periods} because $B_0$ is obviously $M/p_i$-periodic for each $i$. In the Chinese Remainder Theorem geometric representation, we may interpret $A$ as a discrete rectangular box, and $A\oplus B_0$ is a lattice tiling by translates of that box. To construct a non-periodic set tiling complement $B$, we perform a ``column shift" in each direction. Let
$$
a= (p_3^{n-1}-1)M/p_3^{n-1},\ \ b=(p_2^{n-1}-1)M/p_2^{n-1}+ (p_1^{n-1}-1)M/p_1^{n-1}
$$
and define
$$
B(X):= B(X)+ (X^{M/p_1^n}-1) B_1(X)+ (X^{a+ M/p_2^n}-X^a) B_2(X) +  (X^{b+ M/p_3^n}-X^b) B_3(X) .
$$
We still have $A\oplus B=\ZZ_M$, but now $B$ has no periods smaller than $M$. Constructions of this type go back to the work of Szab\'o \cite{Sz}.

Given $0<\beta<3/2$, we need to verify that $M\geq D^{\beta}$ for an appropriate choice of the parameters of the construction. We have
$$
\diam(A)=\frac{(p_1-1)M}{p_1^n} + \frac{(p_2-1)M}{p_2^n} + \frac{(p_3-1)M}{p_3^n} \leq \frac{3M}{p_1^{n-1}} .
$$
Choose a small $\epsilon>0$ to be fixed later, and let
$$
\alpha=\frac{(3-\epsilon)n}{2n+1}.
$$
Note that $0<\alpha<3/2$. Then
\begin{align*}
M&= p_1^np_2^n p_3^n = p_1^{\epsilon n} p_1^\alpha p_1^{n-\alpha-n\epsilon} p_2^n p_3^n
\\
&\geq p_1^{\epsilon n} 2^{-(n-\alpha-n\epsilon)} p_1^\alpha p_2^{(3n-\alpha -n\epsilon)/2} p_3^{(3n-\alpha -n\epsilon)/2} 
\\
&=  p_1^{\epsilon n} 2^{-(n-\alpha-n\epsilon)} 3^{-\alpha} (3p_1p_2^n p_3^n)^\alpha,
\end{align*}
since we chose $\alpha$ so that $n\alpha = (3n-\alpha-n\epsilon)/2$. 

Given $\beta$ with $0<\beta<3/2$, we may choose $\epsilon>0$ small enough and $n$ large enough so that 
$\beta<\alpha<3/2$. We fix these choices. Next,
$$
\frac{p_1^{\epsilon n}}{ 2^{n-\alpha-n\epsilon} 3^{\alpha} }
=
\left(\frac{2}{3}\right)^\alpha 
\frac{p_1^{\epsilon n}}{ 2^{n-n\epsilon}}
\geq 
\left(\frac{2}{3}\right)^{3/2}  
\left( \frac{p_1^\epsilon}{2}\right)^n.
$$
If we choose $p_1$ large enough so that $ (\frac{p_1^\epsilon}{2})^n >(\frac{3}{2})^{3/2}$, we may conclude that
$$
M\geq (3p_1p_2^n p_3^n)^\beta \geq (\diam(A))^\beta
$$
as claimed.

\section{The minimal tiling period and the Coven-Meyerowitz conjecture}
\label{CM-sec}

Coven and Meyerowitz \cite{CM} proved the following theorem.

\begin{theorem}\label{CM-thm} 
Let $S_A$ be the set of prime powers
$p^\alpha$ such that $\Phi_{p^\alpha}(X)$ divides $A(X)$.  Consider the following conditions.

\smallskip
{\it (T1) $A(1)=\prod_{s\in S_A}\Phi_s(1)$,}

\smallskip
{\it (T2) if $s_1,\dots,s_k\in S_A$ are powers of different
primes, then $\Phi_{s_1\dots s_k}(X)$ divides $A(X)$.}
\medskip

Then:

\begin{itemize}

\item[(i)] if $A$ satisfies (T1), (T2), then $A$ tiles $\ZZ$;

\item[(ii)]  if $A$ tiles $\ZZ$ then (T1) holds;

\item[(iii)] if $A$ tiles $\ZZ$ and $|A|$ has at most two distinct prime factors,
then (T2) holds.
\end{itemize}

\end{theorem}

The conjecture that (T2) holds for all finite integer tiles has become known as the {\it Coven-Meyerowitz conjecture}. It remains open in general; see \cite{LL1,LL2,LL3} for recent progress.

The proof of Theorem \ref{CM-thm} (i) in \cite{CM} is by explicit construction. Throughout the rest of this section, we fix
$$
M:={\rm lcm}(S_A)=p_1^{n_1}\dots p_k^{n_k}.
$$
Assuming that $A$ satisfies (T1) and (T2), Coven and Meyerowitz construct an explicit tiling of the integers by $A$ with period $M$.
On the other hand, we always have
$$
D=\diam(A) \geq M/2.
$$
Indeed, we may assume that $\min(A)=0$. Since $\Phi_M|A$ by (T2), we have $A(e^{2\pi i /M})=0$. This would not be possible with $D<M/2$, since then we would have Im$(e^{2\pi i a/M})\geq 0$ for all $a\in A$, with a strict inequality at least once. This implies that if $A$ tiles the integers and satisfies (T2), then it also admits a tiling (possibly a different one) with period $M\leq 2D$.

Coven and Meyerowitz also make the stronger claim without proof (page 167, remarks after Lemma 2.1) that for any set $A$ of nonnegative integers we have $M\leq pD/(p-1)$, or equivalently,
\begin{equation}\label{large-diam}
\diam(A)\geq \frac{p-1}{p}M
\end{equation}
where $p$ is the smallest prime factor of $|A|$. This would imply that if $A$ tiles the integers and satisfies (T2), then it also tiles with period at most $pD/(p-1)$.
However, we could not reproduce their proof and do not believe their claim to be true in general. 

The details are as follows. For general sets of integers (that do not necessarily tile the integers or satisfy (T2)), an easy counterexample is provided by the set $A$ with the mask polynomial
$$
A(X)=\Phi_{p^2}(X)\Phi_{q^2}(X)=(1+X^p+\dots+X^{(p-1)p}) (1+X^q+\dots +X^{(q-1)q}),
$$
where $p,q$ are large primes such that $p<q<2p$.
Then $M=\lcm(S_A)=p^2q^2$, but $D=(p-1)p+(q-1)q$ which is much smaller than $\frac{p-1}{p}M$.

%We verify that $A$ is a set (of course, it does not tile the integers or satisfy (T2)). It suffices to check that the numbers $kp+\ell q$ with $0\leq k\leq p-1$ and $0\leq \ell\leq q-1$ are all distinct. Assume towards contradiction that $kp+\ell q=k'p+\ell' q$, then
%$$
%(k-k')p=(\ell'-\ell)q.
%$$
%Hence $q|(k-k')$. Since $-q<-(p-1)\leq k-k'\leq p-1<q$, we have $k-k'=0$. But then also $\ell=\ell'$, and the claim is proved.
%
%
%\bigskip
%

The more interesting question is whether (\ref{large-diam})
holds if we assume that $A$ satisfies (T1) and (T2). This is in fact true if $|A|$ has at most two distinct prime factors. Indeed, in that case $M$ also has at most two distinct prime factors, say $p$ and $q$. As above, (T2) implies that $\Phi_M|A$. By the 2-prime case of the structure theorem for vanishing sums of roots of unity \cite{deB} (see also \cite[Theorem 3.3 and Corollary 3.4]{LamLeung}), we have
$$
A(X)=P(X) \left(1+X^{M/p}+\dots+X^{(p-1)M/p}\right) + Q(X) \left(1+X^{M/q}+\dots+X^{(q-1)M/q}\right)
$$
mod $(X^M-1)$, for some polynomials $P(X)$ and $Q(X)$ with nonnegative integer coefficients. In other words, $A$ mod $M$ is a union of ``fibers" -- cosets of subgroups of $\ZZ_M$ of order $p$ and $q$. Since the diameter of each such coset is either $(p-1)M/p$ or $(q-1)M/q$, the inequality (\ref{large-diam}) follows. If $|A|$ is a prime power, the same proof applies with the terms involving $q$ removed.

The above argument no longer works when $|A|$ (therefore $M$) has 3 or more distinct prime factors. We still have $\Phi_M|A$, but the corresponding structure theory in this case \cite{deB,Re1,Re2,schoen} is more complicated, and there are many examples of sets $A$ such that $\Phi_M|A$ but $A$ does not contain any fibers (see e.g.~\cite{PR} for a long list of examples). Instead of just the fact that $\Phi_M|A$, we could try to use the full strength of the assumption that $A$ tiles and satisfies (T2); however, there are also examples of sets $A$ that tile the integers, satisfy (T2), satisfy $\Phi_M|A$, and do not contain any fibers \cite{BL}.

A deeper theorem due to Steinberger \cite{Steinberger3} states that (\ref{large-diam}) holds when $2/p_1>1/p_2+\dots+1/p_d$, where the primes dividing $M$ are ordered so that $p_1<\dots<p_d$. In particular, this is always true when $d=3$, since there are only 2 terms on the right side and $p_1<p_2<p_3$. Steinberger does not believe that (\ref{large-diam}) should hold in general, and neither do we.

\bibliographystyle{amsplain}

\begin{thebibliography}{99}

%\bibitem{Bh}
%S. Bhattacharya, {\it Periodicity and Decidability of Tilings of $\ZZ^2$}, Amer. J. Math. 142 (2020),
%255–266.

\bibitem{Biro} A. Bir\'o, {\it Divisibility of integer polynomials and
tilings of the integers}, Acta Arith. 118 (2) (2005), 117--127.

\bibitem{BL} B. Bruce, I. {\L}aba, {\it Keller properties for integer tilings}, preprint, 2024.


\bibitem{CM} E. Coven, A. Meyerowitz, {\it Tiling the integers
with translates of one finite set}, J. Algebra 212 (1999),
161--174.

\bibitem{deB} N.G. de Bruijn, {\it On the factorization of cyclic groups}, Indag. Math. 15 (1953), 370--377.

\bibitem{GT} R. Greenfeld, T. Tao. {\it A counterexample to the periodic tiling conjecture}, Ann. Math., to appear.



\bibitem{Kol} M. N. Kolountzakis, {\it Translational tilings of the integers with long periods}, Electronic J. Comb.10(1) (2003), \# R22.

\bibitem{LL1} I. {\L}aba, I. Londner, {\it Combinatorial and harmonic-analytic methods for integer tilings},
Forum of Mathematics - Pi (2022), Vol. 10:e8, 46 pp.

\bibitem{LL2} I. {\L}aba, I. Londner, {\it The Coven-Meyerowitz tiling conditions for 3 odd prime factors},
Invent. Math. 232(1) (2023), 365-470.


\bibitem{LL3} I. {\L}aba, I. Londner, {\it Splitting for integer tilings and the Coven-Meyerowitz tiling conditions}, 
preprint, 2022.



%\bibitem{LS} J.C. Lagarias, S. Szab\'o, {\it Universal spectra and Tijdeman's
%conjecture on factorization of cyclic groups}, J. Fourier Anal. Appl. 1(7) (2001), 63--70. 
%
%\bibitem{LW1} J.C. Lagarias, Y. Wang, {\it Tiling the line with translates of
%one tile}, Invent. Math. 124 (1996), 341--365.
%
%\bibitem{LW2} J.C. Lagarias, Y. Wang, {\it Spectral sets and factorization of 
%finite abelian groups}, J. Funct. Anal. 145 (1997), 73--98.
%
%
\bibitem{LamLeung} T.Y. Lam and K.H. Leung, {\it On vanishing sums of roots of unity},
J. Algebra 224 (2000), 91--109.

%\bibitem{LM} N. Lev, M. Matolcsi, {\it 	The Fuglede conjecture for convex domains is true in all dimensions}, 
%arXiv:1904.12262,
%to appear in Acta Mathematica.




\bibitem{Mann} H. B. Mann, {\it On Linear Relations Between Roots of Unity}, Mathematika 12, Issue 2 (1965), 107--117.

\bibitem{New} D.J. Newman, {\it Tesselation of integers}, J. Number Theory 9 (1977), 107--111.




\bibitem{PR}
B. Poonen and M. Rubinstein: {\it Number of Intersection Points Made by the Diagonals of a Regular Polygon}, SIAM J. Disc. Math. 11 (1998), 135--156.


\bibitem{Re1} L. R\'edei, {\it \"Uber das Kreisteilungspolynom}, Acta Math. Hungar. 5 (1954), 27--28.

\bibitem{Re2} L. R\'edei, {\it Nat\"urliche Basen des Kreisteilungsk\"orpers}, Abh. Math. Sem. Univ. Hamburg 23 (1959), 180--200.


\bibitem{Sands} A. Sands, {\it On Keller's conjecture for certain cyclic
groups}, Proc. Edinburgh Math. Soc. 2 (1979), 17--21.


%\bibitem{Sands-Hajos} A. Sands, {\it On a conjecture of G. Haj\'os}, Glasgow Math. J. 15 (1974), 88-89.


\bibitem{schoen}
I. J. Schoenberg, {\it A note on the cyclotomic polynomial}, Mathematika 11 (1964), 131-136.




%\bibitem{Steinberger} J. P. Steinberger, {\it Minimal vanishing sums of roots of unity with large coefficients},
%Proc. London Math. Soc. 97 (3) (2008), 689--717.


\bibitem{Steinberger2} J. P. Steinberger, {\it Tilings of the integers can have superpolynomial periods}, Combinatorica 29(4) (2009), 503--509.

\bibitem{Steinberger3} J. P. Steinberger, {\it The lowest-degree polynomial with nonnegative coefficients divisible by the $n$-th cyclotomic polynomial}, Electronic J. Comb. 
19 (4) (2012), Article P1.




\bibitem{Sz} S. Szab\'o, {\it A type of factorization of finite abelian groups},
Discrete Math. 54 (1985), 121--124.



\bibitem{Tij} R. Tijdeman, {\it Decomposition of the integers as a direct sum of
two subsets}, in {\em Number Theory (Paris 1992--1993)}, London Math. Soc.
Lecture Note Ser., vol. 215, 261--276, Cambridge Univ. Press, Cambridge, 1995.


\bibitem{Tij2} R. Tijdeman, {\it Periodicity and almost-periodicity}, with appendix by I. Z. Ruzsa, in: Gy\'ori, E., Katona, G.O.H., Lov\'asz, L., Fleiner, T. (eds), More Sets, Graphs and Numbers. Bolyai Society Mathematical Studies, vol 15. Springer, Berlin, Heidelberg 2006. https://doi.org/10.1007/978-3-540-32439-3-18


\end{thebibliography}

\vfill

\noindent{\sc Department of Mathematics, UBC, Vancouver,
B.C. V6T 1Z2, Canada}

\noindent{\it ilaba@math.ubc.ca}, {ORCID: 0000-0002-3810-6121}

\medskip

\noindent{\sc Department of Mathematics, Massachusetts Institute of Technology, 
Cambridge, MA 02139, USA}

\noindent{\it zakhdm@mit.edu}, ORCID: 0009-0007-6369-7190

\end{document}